\documentclass[12pt]{article}
\usepackage{times}
\usepackage{amsmath}
\usepackage{amsthm,amscd}
\usepackage{amssymb}
\usepackage{latexsym}
\usepackage{enumerate}
\usepackage{enumitem}
\usepackage{graphics}
\usepackage{color}
\usepackage[toc,page]{appendix}
\usepackage{graphicx}
\usepackage{hyperref}
\usepackage{cite}
\usepackage{tikz}
\usepackage{hyperref}
\usepackage{cite}

\newtheorem{theorem}{Theorem}[section]
\newtheorem{lemma}[theorem]{Lemma}
\newtheorem{proposition}[theorem]{Proposition}

\newtheorem{remark}[theorem]{Remark}

\usepackage[square]{natbib}
\renewcommand{\bibname}{References}

\title{More Examples of Non-Rational Adjoint Groups} 
\author{Nivedita Bhaskhar \\ \small{Department of Mathematics \& Computer Science, Emory University, Atlanta, GA 30322, USA.}\\ \href{mailto:nbhaskh@emory.edu}{\small{\tt nbhaskh@emory.edu}}}

\begin{document}
\maketitle

\begin{abstract}
In this paper, we give a recursive construction to produce examples of quadratic forms $q_n$ in the $n^{\mathrm{th}}$ power of the fundamental ideal in the Witt ring whose corresponding adjoint groups $\mathrm{PSO}\left(q_n\right)$ are not stably rational. Computations of the R-equivalence classes of adjoint classical groups by Merkurjev are used to show that these groups are not R-trivial. This extends earlier results of Merkurjev and Gille where the forms considered have non-trivial and trivial discriminants respectively.
\end{abstract}

\section{Introduction}
Rationality of varieties of connected linear algebraic groups has been a topic of great interest. It is well known that in characteristic $0$ every such group variety is rational over the algebraic closure of its field of definition. However, rationality over the field of definition itself is a more delicate question and examples of Chevalley and Serre of non-rational tori and semisimple groups respectively indicate its subtle nature. Platonov's famous example of non-rational groups of the form $\mathrm{SL}_1(D)$ settled negatively the long standing question of whether simply-connected almost simple $k$ groups were rational over $k$, shifting the focus to adjoint groups.

Platonov himself conjectured (\cite{PLR}, pg 426) that adjoint simple algebraic $k$-groups were rational over any infinite field. Some evidence of the veracity of this conjecture is found in \cite{CH} where Chernousov establishes that $\mathrm{PSO}(q)$ is a stably rational $k$ variety for the special quadratic form $q=\langle 1,1,\ldots 1\rangle$ where $k$ is any infinite field of characteristic not $2$. Note that the signed discriminant of the quadratic form in question is $\pm 1$.

However Merkurjev in \cite{ME} constructs a quadratic form $q$ of dimension $6$ with non-trivial signed discriminant over a base field $k$ of characteristic not $2$ and cohomological dimension $2$ such that $\mathrm{PSO}(q)$ is not a $k$-stably rational variety. This example is obtained as a consequence of his computations of \textit{R-equivalence classes} of adjoint classical groups which relates to stable-rationality via the following elementary result :

\textit{If $X$ is a $k$-stably rational variety, then $X(K)/R$ is trivial for any extension $K/k$.}

In fact, Merkurjev shows that if $q$ is a quadratic form of dimension $\leq 6$, then $\mathrm{PSO}(q)(K)/R$ is not trivial for some extension $K/k$ if and only if $q$ is a \textit{virtual Albert form}. Bruno Kahn and Sujatha give a cohomological description of $\mathrm{PSO}(q)(k)/R$ (\cite{BS}, Thm 4) for any virtual Albert form $q$ over fields of characteristic $0$.

It is to be noted that Merkurjev's example uses the non-triviality of the signed discriminant of the quadratic form in a crucial way.

Let $W(k)$ denote the Witt ring of quadratic forms defined over $k$ and $I(k)$ denote the fundamental ideal of even dimensional forms. One can ask if there are examples of quadratic forms $q_n$ defined over fields $k_n$ satisfying the following two properties : 
\begin{enumerate}
\item
$q_n\in I^n(k_n)$, the $n$-th power of the fundamental ideal.
\item
$\mathrm{PSO}(q_n)$ is not $k_n$-stably rational.
\end{enumerate}

Gille answers the question for $n=2$ very precisely in \cite{GI} and produces a quadratic form of dimension 8 with trivial discriminant over a field of cohomological dimension 3. He also shows that the dimension of the quadratic form has to be at least 8 and that the base field should have cohomological dimension at least 2. 

This paper  produces pairs $(k_n,q_n)$ for every $n$ in a recursive fashion such that \\ $\mathrm{PSO}(q_n)(k_n)/R \neq \{1\}$. This implies that $\mathrm{PSO}(q_n)$ is not  stably rational. 

\section{Notations and Conventions}
All fields considered are assumed to have characteristic $0$. Let $W(k)$ denote the Witt ring of quadratic forms defined over $k$ and $I(k)$, the fundamental ideal of even dimensional forms. $P_n(k)$ is the set of isomorphism classes of anisotropic $n$-fold Pfister forms and $I^n(k)$ denotes the $n$-th power of the fundamental ideal. Let us fix the convention that $\langle \langle a \rangle \rangle$ denotes the 1-fold Pfister form $\langle 1,a \rangle$. A generalized Pfister form is any scalar multiple of a Pfister form.

\section{A formula for R-equivalence classes}
Let $G$ be a connected linear algebraic group over $k$. The following relation defined on $G(k)$ is an equivalence relation, called the \textit{R-equivalence} relation.

\[g_0 \sim g_1 \iff \exists \phantom{.}  g(t)\in G(k(t)),\phantom{.} g(0)=g_0 ,\phantom{.} g(1)=g_1 \]

 The induced equivalence classes are called \textit{R-equivalence classes} of $G(k)$.

An  algebraic variety $X$  over $k$ is said to be $k$-stably rational if there exist two affine spaces $\mathbb{A}^n_k, \mathbb{A}^m_k$ and a birational map defined over $k$ between $\mathbb{A}^n_k \times_k X$ and $\mathbb{A}^m_k$.

A $k$ algebraic group $G$ is said to be \textit{R-trivial} if $G(K)/R = \{1\}$ for every extension $K$ of $k$.

Recall the fact that stably $k$ rational varieties are R-trivial. The following formula, a special case of Merkurjev's computations of R-equivalence classes of classical adjoint groups, is a key ingredient for our construction
of  nonrational adjoint groups.

\begin{theorem}[Merkurjev, \cite{ME}, Thm 1]
\label{Reqformula}$\mathrm{PSO}(q)(k)/R \cong G(q)/\mathrm{Hyp}(q)k^{\times 2}$  where 
\begin{enumerate}
\item
$q$ is an even dimensional (say of dimension $2b$) non-degenerate quadratic form over $k$ of characteristic $\neq 2$.
\item
$G(q) = \{a\in k^{\times}|\phantom{.} aq\cong q \}$, the group of similarities.
\item
$\mathrm{Hyp}(q) = \langle N_{l/k}(l^{\times})|\phantom{.}[l:k]< \infty, q_l \cong \mathbb{H}^{b} \rangle $. 
\item
$k^{\times 2}:= \{a^2|\phantom{.}a\in k^{\times}\}$.
\end{enumerate} 
\end{theorem} 

Thus to check that $\mathrm{PSO}\left(q_n\right)$ is not $k_n$-stably rational, it is enough to check that \\  $\mathrm{PSO}\left(q_n\right)\left(k_n\right)/R\neq \{1\}$ using the above formula. 

\section{Lemmata}
This section collects a list of lemmas which come in handy whilst constructing nonrational adjoint groups.

\begin{lemma}[Odd extensions]\label{odd}Let $q$ be a quadratic form over $k$. Let $k'/k$ be an odd degree extension.  Then, 
\[\mathrm{PSO}\left(q_{k'}\right)(k')/R = \{1\} \implies \mathrm{PSO}({q})(k)/R =\{1\}.\]
\end{lemma}

\begin{proof}
Suppose that $x\in G(q)$.

Clearly $G(q)\subseteq G\left(q_{k'}\right)=\mathrm{Hyp}\left(q_{k'}\right)k'^{\times 2}$ as $\mathrm{PSO}\left({q}_{k'}\right)(k')/R=\{1\}$. 

The definition of $\mathrm{Hyp}$ groups and the transitivity of norms immediately yield the fact that $N_{k'/k}\left(\mathrm{Hyp}\left(q_{k'} \right)k'^{\times 2}\right)\subseteq \mathrm{Hyp}({q})k^{\times 2}$.

If $2n+1$ is the degree of $k'$ over $k$, it follows that  $x^{2n+1}=N_{k'/k}(x)\in \mathrm{Hyp}({q})k^{\times 2}$. Hence $x\in \mathrm{Hyp}({q})k^{\times 2}$.
\end{proof}

Let $p$ be a Pfister form over $k$. Its \textit{pure-subform} $\tilde{p}$ is defined uniquely upto isometry via the property that $\tilde{p}\perp \langle 1\rangle \cong p$. The following useful result connects the values of pure-subforms and Pfister forms : 
  
\begin{lemma}[\cite{SC}, Chap 4, Thm 1.4]
\label{puresubform}
If $D(q)$ denotes  the set of non-zero values represented by the quadratic form $q$, then, for $p\in P_n(k)$, 

\[b\in D(\tilde{p}) \iff p \cong \langle \langle b, b_2, \ldots, b_n \rangle\rangle, \]

for some $b_2, \ldots, b_n\in k^{\times}$. 
\end{lemma}

The next lemma is a useful tool for converting an element which is a norm from two different quadratic extensions into a norm from a biquadratic extension of the base field upto squares.

\begin{lemma}[Biquadratic-norm trick, \cite{KLST}, Lemma 1.4]
\label{biquadratic}
If $l_1$ and $l_2$ are two quadratic extensions of a field $l$, then

\[N_{l_1/l}\left(l_1^{\times}\right)\cap N_{l_2/l}\left(l_2^{\times}\right)=N_{l_1\otimes_l l_2/l}\left(\left(l_1\otimes_l l_2\right)^{\times}\right)l^{\times 2}.\]
\end{lemma}

\begin{lemma}[Folklore]\label{folklore}
Let $k(u)$ be a finite separable extension of $k$ generated by $u$ of degree $p^gh$ where $p$ is a prime not dividing $h$ and $g\geq 1$. Then there exist finite separable extensions $M_1/M_2/k$ such that the following conditions hold : 
\begin{enumerate}
\item
$k(u)\subset M_1$ and $M_1 = M_2(u)$.
\item
$[M_1 : M_2]=p$ and  $p\not |\phantom{.} [M_1:k(u)]$
\end{enumerate}
\end{lemma}

\begin{proof}
Let $M/k$ be any finite Galois extension containing $k(u)$ and let $S$ be any $p$-Sylow of $\mathrm{Gal}(M/k(u))$. Since $\mathrm{Gal}(M/k(u))$ is a subgroup of $\mathrm{Gal}(M/k)$, there is a $p$-Sylow subgroup $T$ of $\mathrm{Gal}(M/k)$ containing $S$. 

Let $M^S$ and $M^T$ denote the fixed fields of $S$ and $T$ in $M$ respectively. Note that $M^S \supseteq M^T$ and since $S$ and $T$ are appropriate $p$-Sylow subgroups, we have $p\not |\phantom{.} [M^S : k(u)]$ and $p\not |\phantom{.} [M^T : k]$. Comparing degrees yields $[M^S : M^T]=p^g$. Also note that $u\not\in M^T$ and $k(u)\subset M^S$.
 
In fact, $M^T(u)=M^S$ because $[M^S:M^T]$ and $[M^S:k(u)]$ are coprime. 

$S$ is a proper subgroup of its normalizer $N_T(S)$ because $T$ is nilpotent and $S\neq T$. Thus, you can find a subgroup $V$ such that $S\subseteq V\subseteq T$ and index of $S$ in $V$ is $p$. Set $M_2$ to be the fixed field of $V$ in $M$.

Thus $M_1 = M_{2}(u)$ is of degree $p$ over $M_2$ and satisfies the other conditions given in the Lemma.

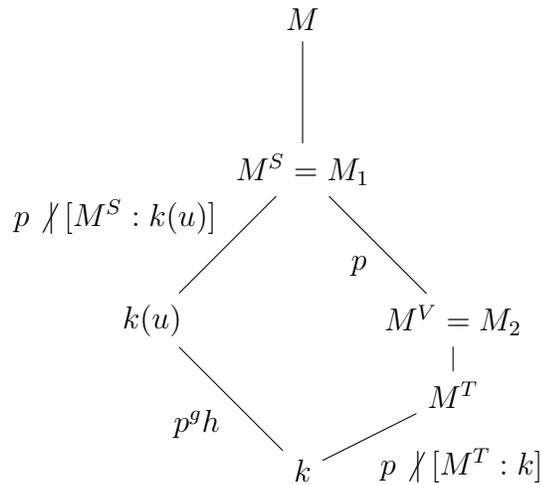
\begin{center}
\begin{tikzpicture}[node distance = 2cm, auto]
\node (Q) {$k$};
\node (E) [node distance = 2cm, above of=Q, left of=Q] {$k(u)$};
\node (F) [node distance = 1cm, above of=Q, right of=Q, right of=Q] {$M^T$};
\node (G) [node distance = 1cm, above of=F ]{$M^V=M_2$};
\node (K) [above of=Q, node distance = 4cm] {$M^S=M_1$};
\node (L) [above of=K ]{$M$};
\draw[-] (Q) to node {$p^gh$} (E);
\draw[-] (Q) to node [swap] {$p\not | \phantom{.}[M^T:k]$} (F);
\draw[-] (E) to node {$p\not | \phantom{.}[M^S:k(u)]$} (K);
\draw[-] (F) to node [swap] {} (G);
\draw[-] (G) to node {$p$} (K);
\draw[-] (K) to node {} (L);
\end{tikzpicture}
\end{center}

\end{proof}

The following lemma tells us that Pfister forms yield R-trivial varieties. Note that in fact more is true, namely that $\mathrm{PSO}(q)$ is stably-rational for any generalized Pfister form $q$ (\cite{ME}, Prop 7).

\begin{lemma}\label{pfister}If $q$ is an $n$-fold Pfister form over a field $k$, then \[\mathrm{PSO}(q)(k)/R=\{1\}.\]
\end{lemma}

\begin{proof}

If $q$ is isotropic, it is hyperbolic and hence $G(q)=\mathrm{Hyp}(q)=k^{\times}$. Therefore assume without loss of generality that $q$ is anisotropic. 

\textbf{Case $n=1$ : }  Let $q = \langle 1,-a \rangle$. Then $G(q)= N_{k\left(\sqrt{a}\right)/k}(k\left(\sqrt{a}\right)^{\times})$. Further, $q$ splits over  a finite field extension $L$ of $k$  if and only if $a$ is a square in $L$. Therefore $q_L$ splits if and only if $L\supseteq k\left(\sqrt{a}\right) \supseteq k$ and hence clearly $\mathrm{Hyp}(q) = G(q)$.

\textbf{General case : } Recall that Pfister forms are round, that is $D(q)=G(q)$ for any Pfister form $q$. Let $\tilde{q}$ be the pure-subform of $q$. If $b\in D(\tilde{q})\subseteq D(q)$, then by Lemma \ref{puresubform}

\[b \in D (\langle 1, b \rangle) = \mathrm{Hyp}\left(\langle 1,b\rangle\right) k^{\times 2} \subseteq \mathrm{Hyp}(q) k^{\times 2}.\]

Note that any $x\in G(q)=D(q)$ can be written (upto squares from $k^{\times}$) as either $b$ or $1+b$ for some $b\in D(\tilde{q})$. Since $x=b\in D(\tilde{q})$ has just been taken care of, it is enough to note that for $b\in D(\tilde{q})$,

\[1+b\in D\left(\langle 1,b\rangle\right)\subseteq \mathrm{Hyp}(q)k^{\times 2}.\]
\end{proof}

\section{Comparison of some Hyp groups}
Let $q$ be an anisotropic quadratic form over  a field  $k$ of characteristic $0$. Let $p$ be an anisotropic Pfister form defined over $k$ and let $Q=q\perp tp$ over the field of Laurent series $K=k((t))$. Note that $K$ is a complete discrete valued field with uniformizing parameter $t$ and residue field $k$. Recall the exact sequence  in Witt groups :

\[0\to W(k)\xrightarrow{Res} W(K)\xrightarrow{\delta_{2,t}} W(k)\to 0 \]

where $Res$ is the restriction map and  $\delta_{2,t}$ denotes the second residue homomorphism
with respect to the parameter $t$.

\begin{remark}
\label{remark}
$Q$ is anisotropic and $\mathrm{dim}(Q) > \mathrm{dim}(p)$.
\end{remark}
This can be shown with the aid of the above exact sequence. Let the anisotropic part of $Q$ be $Q_{an} \cong q_1 \perp t q_2$ for quadratic forms $q_i$ defined over $k$. Then each $q_i$ is anisotropic. The following equality in $W(k)$ is in fact an isometry because the forms are anisotropic :

\[\delta_{2,t}(Q) = p = q_2.\]

This immediately implies $q\cong q_1$. The inequality between dimensions of $Q$ and $p$ follow immediately.

\begin{proposition}\label{comp2}
$\mathrm{Hyp}(Q)K^{\times 2}\subseteq \mathrm{Hyp}\left(q_K\right)K^{\times 2}$  if  $\mathrm{PSO}(q)(k)/R\neq \{1\}$.
\end{proposition}

\begin{proof}
Let $L/K$ be a finite field extension which splits $Q$. There is a unique extension of the discrete valuation on $K$ to $L$ which makes $L$ into a complete discrete valued field. Let $l$ denote the residue field of $L$. Since the characteristic of $k$ is $0$, $k\subseteq K$ and $l\subseteq L$. Let $K_{nr}$ denote the maximal non-ramified extension of $K$ in $L$ and $\pi$ be a uniformizing parameter of $L$. Let $f = [l:k]$, the degree of the residue field extensions and $e$ be the ramification index of $L/K$. Let $v_{X}$ denote the corresponding valuation on fields $X=K, K_{nr}, L$ and $O_X$, the corresponding discrete valuation rings.

Since $L/K_{nr}$ is totally ramified, the minimal polynomial of $\pi$ (which is also its characteristic polynomial) over $K_{nr}$ is an \textit{Eisenstein} polynomial $x^e + a_{e-1}x^{e-1}+ \ldots + a_1 x + a_0$ in $K_{nr}[x]$, where $v_{K_{nr}}\left( a_0 \right)=1$ and $v_{K_{nr}}\left(a_i\right)\geq 1\phantom{.}\forall\phantom{.} 1\leq i\leq e-1$ (\cite{CF}, Chap 1, Sec 6, Thm 1). Note that $N_{L/K_{nr}}(\pi)=(-1)^ea_0$.

$K_{nr}=l((t))$ and $L=l((\pi))$ (\cite{SE}, Chap 2, Thm 2). Let $a_0 = -ut(1+u_1t + \ldots)$ in $O_{K_{nr}}=l[[t]]$. By Hensel's lemma, $1+u_1t + \ldots = w^2$ for some $w$ in $K_{nr}$. The relation given by the Eisenstein polynomial can be rewritten by applying Hensel's lemma again as follows :  

\[\pi^e = utv^2 , u\in l^{\times}, v\in L^{\times}.\]

Hence the norm of $\pi$ can be computed upto squares. That is,

\[N_{L/K}(\pi)= N_{K_{nr}/K}((-1)^ea_0) \in (-1)^{ef}(-t)^fN_{l/k}(u)K^{\times 2}.\]

The problem is subdivided into two cases depending on the parity of the ramification index $e$ of $L/K$. 

\subsubsection*{Case I : $e$ is odd}
We show that $L$ also splits $q$ in this case. Let $\delta_{2,\pi} : W(L)\to W(l)$ be the second Milnor residue map with respect to the uniformizing parameter $\pi$ chosen above. Note that $Q_L = q + \pi up$ in $W(L)$. Then 

 \begin{align*}
 Q_L = 0 &\implies \delta_{2,\pi}\left(Q\right)= 0 \in W(l) \\
 & \implies up=0 \in W(L) \\
 & \implies q=0 \in W(L)
\end{align*}

\subsubsection*{Case II : $e$ is even}
Now $Q_L = q+up=0$ in $ W\left(L\right)$. Since any element of $L^{\times}$ is of the form $\alpha \pi b^2$ or $\alpha b^2$ for some $\alpha \in l^{\times}$ and $b\in L$, the norm computation of $\pi$ done before yields the following :  

\[N_{L/K}\left(L^{\times}\right) \subseteq \langle N_{l/k}\left(u\right)(-t)^{f}\rangle K^{\times 2}.\]

So it is enough to show that $f$ is even and $N_{l/k}\left(u\right)$ is in $\mathrm{Hyp}\left(q\right)k^{\times 2}$.

\textbf{Claim} : $f$ is even.

If $f$ is odd, then $\mathrm{PSO}\left(q_l\right)\left(l\right)/R\neq \{1\}$ by Lemma \ref{odd}. But $q_l = -up$ is a form similar to a Pfister form. Hence by Lemma \ref{pfister}, $\mathrm{PSO}\left(q_l\right)(l)/R = \{1\}$ which is a contradiction.

\textbf{Claim} : $N_{l/k}\left(u\right)\in \mathrm{Hyp}\left(q\right)k^{\times 2}$

Look at $l\supseteq k\left(u\right)\supseteq k$. If $[l:k\left(u\right)]$ is even, then $N_{l/k}\left(u\right)=N_{k\left(u\right)/k}\left(u^{[l:k\left(u\right)]}\right)\in k^{\times 2}$ which proves the claim. 

Otherwise $r: W\left(k\left(u\right)\right)\to W\left(l\right)$ is injective and hence $q+up=0$ in $W\left(k\left(u\right)\right)$. It remains to show that $N_{k\left(u\right)/k}\left(u\right)\in \mathrm{Hyp}\left(q\right)k^{\times 2}$. 

Suppose that $[k(u):k]$ is odd. Then Lemma \ref{odd} implies that $\mathrm{PSO}\left(q_{k(u)}\right)/R \neq \{1\}$. On the other hand, $q_{k(u)}$ is similar to Pfister form $p_{k(u)}$. This contradicts Lemma \ref{pfister}. Therefore $[k(u):k]$ is even.

Let $[k(u):k]=2^g h$ where $h$ is odd and $g\geq 1$. Lemma \ref{folklore} gives us a quadratic extension $M_1=M_2(u)$ over $M_2$ such that $M_1$ is an odd extension of $k(u)$. 

Since $[M_1:k(u)]$ is odd, there is a $w\in k^{\times}$ such that 

\[N_{M_1/k}(u) = N_{k(u)/k}\left(N_{M_1/k(u)}(u)\right)= N_{k(u)/k}(u)w^2.\] 

Hence it suffices to show that $N_{M_1/k}(u)\in \mathrm{Hyp}(q)k^{\times 2}$. Using transitivity of norms and the definition of $\mathrm{Hyp}$ groups, showing $N_{M_1/M_2}\left(u\right)\in \mathrm{Hyp}\left(q_{M_2}\right){M_2}^{\times 2}$ proves the claim. 

Let $\eta:= N_{M_1/M_2}\left(u\right)$. By using Scharlau's transfer and Frobenius reciprocity (\cite{SC}, Chap 2, Lemma 5.8 and Thm 5.6),

\[ p\otimes \langle 1 \rangle  = - \langle u \rangle q \in W\left(M_1\right) \implies p\otimes \langle 1, -\eta \rangle = 0 \in W\left(M_2\right).\]

Hence $\eta \in G\left( p_{M_2}\right)=D\left(p_{M_2}\right)$ since $p$ is a Pfister form.

Let $s$ be the pure subform associated with $p_{M_2}$. We can assume (upto squares from $M_2$) that $\eta=b$ or $1+b$ for some $b\in D\left(s\right)$. In either case, $\eta\in N_{M_2\left(\sqrt{-b}\right)/M_2}\left(\left(M_2\left(\sqrt{-b}\right)\right)^{\times}\right)$. By Lemma \ref{puresubform}, $p = \langle \langle b, \ldots \rangle\rangle $. Note that if $-b$ is already a square in $M_2$, then the above reasoning shows that $q$ splits over $M_1$ which shows that $\eta \in \mathrm{Hyp}\left(q_{M_2}\right)M_2^{\times 2}$. If $-b$ is not a square, then $p$ splits in $M_2\left(\sqrt{-b}\right)$ and hence $q=-up$ splits in $M_1\left(\sqrt{-b}\right)$. 

The introduction of subfield $M_2$ is useful because the biquadratic norm trick can be used ! More precisely, since  

\[\eta \in N_{M_2\left(\sqrt{-b}\right)/M_2}\left(\left(M_2\left(\sqrt{-b}\right)\right)^{\times}\right)\cap N_{M_1/M_2}\left(M_1^{\times}\right),\]

Lemma \ref{biquadratic} shows that $\eta$ is upto squares a norm from $M_1\left(\sqrt{-b}\right)$ and $M_1\left(\sqrt{-b}\right)$ splits $q$. Thus $\eta\in \mathrm{Hyp}\left(q_{M_2}\right)M_2^{\times 2}$ as claimed.   
\end{proof}

\begin{proposition}\label{comp1} 
$\mathrm{Hyp}\left(q_K\right)K^{\times 2}\subseteq \mathrm{Hyp}\left(q\right)K^{\times 2}.$
\end{proposition}

\begin{proof}
Using the exact sequence associated to the second Milnor residue map again, it is clear that if $q$ is split by a finite field extension $L$ of $K$, then it is also split by $l$, the residue field of $L$. Thus $\mathrm{Hyp}\left(q_K\right)$ is generated by $N_{L/K}\left(L^{\times}\right)$ where $L$ runs over finite unramified extensions of $K$ which split $q$. By Springer's theorem, $[l:k]$ has to be even. And characteristic of $k = 0$ implies that $L\cong l((t))$. To conclude, it is enough to observe that 

\[N_{l((t))/k((t))}\left( l((t)) \right)^{\times} \subseteq N_{l/k}(l^{\times})K^{\times 2}.\]

\end{proof}

\section{A recursive procedure}
All fields have characteristic $0$. We say a quadruple $(n,\lambda,L,\phi)$ \textit{has property $\star$} if the following holds : 

$\phi$ is an anisotropic quadratic form over $L$ in $I^{n}\left(L\right)$ such that the scalar $\lambda$ is in $G\left(\phi\right)$ but not in $\mathrm{Hyp}\left(\phi\right){L}^{\times 2}$ and there exists a decomposition of $\phi$ into a sum of generalized $n$-fold Pfister forms in the Witt ring $ W\left(L\right)$, each of which is annihilated by $\langle 1,-\lambda \rangle$. More precisely, in $W\left(L\right)$, 
\begin{align*}
\phi = \sum_{i=1}^{m} \alpha_i p_{i,n} ,\mathrm{where}\phantom{.}\alpha_i \in {L}^{\times}&, p_{i,n}\in P_n\left(L\right) \\
\langle 1,-\lambda \rangle\otimes p_{i,n} = 0 \phantom{.}&\forall\phantom{.} i.
\end{align*}

Assume that $(n,\lambda,k_n,q_n)$ has property $\star$ with $q_n = \sum_{i=1}^{m} \alpha_i p_{i,n}$ for $p_{i,n}\in P_n\left(k_n\right)$ and $\alpha_i\in k_n^{\times}$ such that each $p_{i,n}$ is annihilated by $\langle 1, -\lambda \rangle$. Let $K_0$ denote the field $k_n$. Define the fields $K_i$ recursively as follows : \[K_i := K_{i-1}\left(\left(t_i\right)\right)\phantom{.} \forall\phantom{.} 1\leq i\leq m \]

Let $Q_0$ denote the quadratic form $q_n$ defined over $K_0$. Define the quadratic forms $Q_i$ over fields $K_i$ recursively as follows : \[ Q_i := Q_{i-1} \perp t_ip_{i,n}\phantom{.}\forall\phantom{.} 1\leq i\leq m \]

Note that $\lambda\in G\left(Q_i\right)$ for each $1\leq i \leq m$ since $\lambda \in G\left(q_n\right)$ and $G\left(p_{i,n}\right)$ for each $i$.

\begin{theorem}\label{recursive}
Let $(n,\lambda,k_n,q_n)$ has property $\star$. Then for $(K_m,Q_m)$ as above, the following hold :  
\begin{enumerate}
\item
$Q_m\in I^{n+1}\left(K_m\right)$
\item 
$\lambda \in  G\left(Q_m\right)\setminus \mathrm{Hyp}\left(Q_m\right){K_m}^{\times 2}$. In particular, $\mathrm{PSO}\left(Q_m\right)$ is not $K_m$-stably rational.
\item
$\left(n+1,\lambda,K_m,Q_m\right)$ has property $\star$. 
\end{enumerate}
\end{theorem}

\begin{proof}

In the Witt ring $W\left(K_m\right)$, 

\begin{equation}
\label{equation1}Q_m = q_n + \sum_{i=1}^m t_i p_{i,n} = \sum_{i=1}^{m} \alpha_i p_{i,n} + t_i p_{i,n} = \sum_{i=1}^m  p_{i,n}\otimes \langle \alpha_i, t_i \rangle \in I^nI \subseteq I^{n+1}\left( K_m\right)\phantom{.}-
\end{equation}

We now prove by induction that $\lambda \in G\left(Q_i\right)\setminus \mathrm{Hyp}\left(Q_i\right){K_i}^{\times 2}$ for each $i\leq m$.

The base case $i=0$ is given, namely the pair $\left(k_n,q_n\right)$. Assume as induction hypothesis that this statement holds for all $i\leq j$. The proof of the statement for $i=j+1$ follows :  

The following notations are introduced for convenience.  
\begin{align*}
\left(Q,K\right)&:=\left(Q_{j+1},K_{j+1}\right) \\
\left(q,k\right)&:=\left(Q_{j},K_{j}\right) \\ 
t &:=t_{j+1}  \\
p &:= p_{j+1,n}\in P_n\left(k\right)
\end{align*}

Thus $Q=q+tp \in W\left(K\right)$. 

Since $\lambda\in k^{\times}$ and not in $\mathrm{Hyp}(q)k^{\times 2}$, it is not in $\mathrm{Hyp}(q)K^{\times 2}$. By Proposition \ref{comp1}, $\lambda\not\in \mathrm{Hyp}(q_K)K^{\times 2}$ and by Proposition \ref{comp2}, $\lambda \not\in\mathrm{Hyp}(Q)K^{\times 2}.$ By construction, $\lambda \in G(Q)$ as $\lambda \in G\left(p\right)\cap G(q)$. Hence $\lambda\in G(Q)\setminus \mathrm{Hyp}(Q)K^{\times 2}$.

It is clear that $\left(n+1,\lambda,K_m,Q_m\right)$ has property $\star$ by Equation (\ref{equation1}).
\end{proof}

\section{Conclusion}

\begin{theorem} For each $n$, there exists a quadratic form $q_n$ defined over a field $k_n$ such that $\mathrm{PSO}\left(q_n\right)$ is not $k_n$-stably rational. 
\end{theorem} 

\begin{proof}
Let $q$ be an anisotropic quadratic form of dimension $6$ over a field $F$ of characteristic 0. If the discriminant of $q$ is not trivial and $C_{0}(q)$ is a division algebra, then there exists a field extension $E$ of $F$ such that $\mathrm{PSO}(q)(E)/R \neq \{1\}$ (\cite {ME}, Thm 3). 

Define $k_1 := E$, $q_1 : = q_E$ and pick a $\lambda\in G\left(q_1\right)\setminus \mathrm{Hyp}\left(q_1\right)k_1^{\times 2}$.

We can write $q_1 = \sum_{i=1}^{r}\alpha_if_i$ in the Witt ring $W\left(k_1\right)$ for some scalars $\alpha_i\in k_1^{\times}$ and $1$-fold Pfister forms $f_i$ which are annihilated by $\langle 1, -\lambda \rangle$(\cite{SC}, Chap 2, Thm 10.13).

Therefore Theorem \ref{recursive} can be applied repeatedly to produce pairs $\left(k_n,q_n\right)$ such that  \[\mathrm{PSO}\left(q_n\right)\left(k_n\right)/R\neq \{1\}.\] 
This implies that $\mathrm{PSO}\left(q_n\right)$ is not $k_n$-stably rational.
\end{proof}

\textit{Acknowledgements} : The author thanks Professors P. Gille, A.S Merkurjev and R. Parimala for their valuable suggestions and critical comments.

\end{document}